\documentclass{article}
\usepackage{graphicx}

\usepackage[fleqn]{amsmath}

\usepackage{amsthm}
\usepackage{amssymb}
\usepackage{amsbsy}
\usepackage{amsfonts}
\usepackage{mathrsfs}
\usepackage{amsmath}
\usepackage[all]{xy}
\usepackage{amstext}
\usepackage{amscd}
\usepackage[dvips]{epsfig}
\usepackage{psfrag}
\usepackage{enumerate}
\usepackage{flafter}
\allowdisplaybreaks

\theoremstyle{plain}
\newtheorem{thm}{Theorem}[section]
\newtheorem{prop}[thm]{Proposition}

\newtheorem{lem}[thm]{Lemma}
\theoremstyle{definition}
\newtheorem{exa}[thm]{Example}

\newtheorem{rem}[thm]{Remark}
\newtheorem{defn}[thm]{Definition}
\newtheorem{prob}[thm]{Problem}

\def\Hom{\mathop{\mathrm{Hom}}\nolimits}

\def\F{\mathop{\mathbb{F}}\nolimits}

\newcommand{\lra}{\longrightarrow}
\newcommand{\ra}{\rightarrow}
\newcommand{\Q}{{\Bbb Q}}

\newcommand{\Z}{{\Bbb Z}}

\newcommand{\Ann}{{\rm Ann }}

\newcommand{\X}{{\tilde{X}}}

\newcommand{\id}{{\mathrm{id}}}

\newcommand{\pcc}[2]{\mbox{$\begin{array}{c}
\includegraphics[scale=#2]{#1.pdf}
\end{array}$}}

\begin{document}
\large
\begin{center}
{\bf\Large Skew-rack cocycle invariants of closed 3-manifolds}
\end{center}
\vskip 1.5pc
\begin{center}{\Large Takefumi Nosaka}\end{center}
\vskip 1pc\begin{abstract}\baselineskip=12pt \noindent
We establish a new approach to obtain 3-manifold invariants via Dehn surgery.
For this, we introduce skew-racks with good involution and Property FR, and define
cocycle invariants as 3-manifold invariants.

\end{abstract}

\begin{center}
\normalsize
\baselineskip=11pt
{\bf Keywords} \\
\ \ \ 3-manifolds, surgery, knots, birack \ \ \
\end{center}
\begin{center}
\normalsize
\baselineskip=11pt
{\bf Subject Codes } \\
\ \ \ 57M27, 20J06, 16T25, 19J25\ \ \
\end{center}

\baselineskip=13pt

\large
\baselineskip=16pt

\section{Introduction}
Every closed 3-manifold $M$ with orientation can be obtained from a framed link in the 3-sphere $S^3$ by a Dehn surgery.
Since there is a one-to-one correspondence between closed
3-manifolds and framed links in $S^3$ modulo either the Kirby moves \cite{Kirby} or Fenn-Rourke moves \cite{FR},
any framed link invariant, which is invariant with respect to the moves, is a 3-manifold invariant.
For example, in quantum topology, frameworks from Chern-Simons theory have produced many 3-manifold invariants,
including the concepts of modular categories 
{\it etc}; see, e.g., \cite{RT, Tur} and references therein.
In contrast, from classical views as in algebraic topology, 
the fundamental groups, $\pi_1(M)$, of 3-manifolds, have useful information, and are strong invariants;
furthermore, as in the Dijkgraaf-Witten model \cite{DW}, starting from a finite group $G$, we can
define a certain weight of the set, $\mathrm{Hom}(\pi_1(M), G)$, in terms of group cohomology of $G$.
However, apart from the quantum invariants and fundamental groups, there are few procedures
to yield 3-manifold invariants via Dehn surgery.

In this paper, we establish a new approach to yield 3-manifold invariants in a classical situation.
For this, we focus on a class of skew-racks (see Section \ref{gg2}), which is an algebraic system and a modification of quandles and biracks.
As in quandle theory (see, e.g., \cite{FRS, CEGN, CJKLS, Nosbook}), starting from skew-racks, we can also define a set of
colorings of framed links and weights of the set, where the weights are evaluated by birack 2-cocycles, and
are called {\it a cocycle invariant} as a framed link invariant (see Section \ref{gg2} for details).
The point of this paper is to find skew-racks such that
the cocycle invariant is stable under the Fenn-Rourke moves.
Following the point, we define Property FR of skew-racks (Definition \ref{def4}),
and show (Theorem \ref{m3nthm4} and Proposition \ref{prop8}) that, in some situations, the associated cocycle invariant gives rise to a 3-manifold invariant.
In Section \ref{SCGSS}, we establish several examples of skew-racks with Property FR; for example,
from a group $G$ and an involutive automorphism $\kappa: G \ra G$, we can define a skew-racks with Property FR (Examples \ref{ss2} and \ref{exa28}).

Using the examples of skew-racks, we compute the set of colorings and
compute some cocycle invariants. For example, we determine the invariants of some Brieskorn 3-manifolds as integral homology 3-spheres 
(Example \ref{exa3355}).
Following the computations, we discuss a comparison with the Dijkgraaf-Witten invariant and pose some problems (Problem \ref{p3p64}).
Finally, we attempt to make an application from the skew-racks above;
Precisely, Section \ref{SCG211} suggests elementary approaches to find some 3-manifolds, which are not the results of surgery of any knot in $S^3$; however, unfortunately, in this paper,
we find no examples of their applications.








\

\noindent
{\bf Conventional notation.}
Every 3-manifold is understood to be connected, smooth, oriented, and closed.

\subsection*{Acknowledgment}
The author is grateful to Nozomu Sekino, Kimihiko Motegi and Motoo Tange for giving him valuable comments on Dehn surgery.



\section{Symmetric skew-racks and birack cocycle invariants}\label{gg233}
We introduce skew-racks, as a special class of biracks (see \cite{FR, CEGN} for the definition of biracks). We define {\it a skew-rack} to be a triple of a set $X$, a binary operation $\lhd: X \times X \ra X $, and a bijection $\kappa : X \ra X$ satisfying
\begin{enumerate}[(SR1)]
\item For any $a,b \in X$, the equality $\kappa (a \lhd b)= \kappa (a) \lhd \kappa( b)$ holds.
\item For any $b \in X$, the map $X \ra X$ that sends $x$ to $x\lhd b$ is a bijection.
\item For any $a,b,c \in X$, the distributive law $(a\lhd b) \lhd c=(a \lhd \kappa(c))\lhd (b \lhd c)$ holds. 
\end{enumerate}
As a special case, if $\kappa = \mathrm{id}_X $, the definition of skew-racks coincides with that of racks.
We often denote the inverse map $\bullet \lhd b$ of the bijection by $ \bullet \lhd^{-1} b.$
Furthermore,
as a slight generalization of symmetric quandles in \cite{Kam,KO}, we define {\it a symmetric skew-rack} to be a pair of a skew-rack $(X,\lhd, \kappa) $
and an involution $\rho:X \ra X$ satisfying
\begin{enumerate}[(SS1)]
\item For any $a, b \in X $, the equalities $(a \lhd b ) \lhd \rho (b)= a $ and $ \rho(a) \lhd \kappa(b) =\rho( a \lhd b)$ hold.
\item The involutivity $\rho \circ \rho = \kappa \circ \kappa = \mathrm{id}_X$ and the commutativity $ \rho \circ \kappa = \kappa \circ \rho$ hold.
\end{enumerate}
Such a $\rho$ is called {\it a good involution} as in \cite{Kam}.
If $\kappa = \mathrm{id}_X $ and the equality $a \lhd a =a$ holds for any $a\in X$,
the definition of symmetric biracks is exactly the original definition of symmetric quandles \cite{Kam}.
Let us give examples of symmetric skew-racks.

\begin{exa}\label{ss1} Let $X$ be a group $G$, and $\kappa : G \ra G$ be an involutive automorphism.
Define $x \lhd y$ by $ \kappa (y^{-1}) xy$, and $\rho(x)$ by $ x^{-1}$.
Then, these maps define a symmetric skew-rack structure on $X$.
\end{exa}
\begin{exa}\label{ss2}
Let $K$ be a group, and $f: K \ra K$ be an involutive automorphism.
Consider the direct products $X= K \times K$ and $\kappa = f\times f$.
Define $(x,a) \lhd (y,b)$ by $ (f(x) y^{-1} by,f(a) ) $, and $\rho(x,a)= ( f(x), f(a)^{-1})$.
Then, these $X,\lhd, \kappa, \rho$ define a symmetric skew-rack structure on $K \times K$,
and $ \mathrm{Tw}(x,a) =(a^{-1}x,a) $.
As seen in Sections \ref{SCG2}--\ref{SCG211}, this skew-rack plays a key role in this paper.
\end{exa}
\noindent
Finally, we see that, if $X$ is of finite order, the axiom (SS3) is obtained from the other axioms. 
\begin{prop}\label{propss2}Let $(X,\lhd, \kappa) $ be a skew-rack satisfying $ \kappa^2 = \mathrm{id}_X $ as in (SS2).
Define the map $\mathrm{Tw}: X \ra X$ by setting $ \mathrm{Tw}(x) = \kappa(x)\lhd^{-1}\kappa(x)$ .
Then, the map is bijective, where the 
inverse is the map $X \ra X $ that sends $ x $ to $\kappa (x) \lhd x$. 
\end{prop}
\begin{proof}
Letting $y$ be $ \mathrm{Tw} ( \kappa (x)\lhd x  )$, 
we may show $y= x$. Notice that $x\lhd \kappa (x)= y \lhd (x \lhd \kappa (x))$, which is equal to 
\[  \bigl(( y\lhd^{-1}x) \lhd x \bigr) \lhd (x \lhd \kappa (x))= \bigl(( y\lhd^{-1}x) \lhd x \bigr)\lhd \kappa (x) =y \lhd \kappa (x). \] 
Thus, by (SS2), we have $y=x$. Similarly, we can easily check $\kappa (\mathrm{Tw} ( x))\lhd \mathrm{Tw} ( x) = x$, leading to the 
proof. 
\end{proof}

\section{Preliminaries; colorings and birack cocycle invariants}\label{gg2}
We will define $X$-colorings, although
the definition may be seen as a slight modification of classical $X$-colorings of quandles or biracks; see \cite{CJKLS, CEGN, FRS4}.
Let $D$ be a framed link diagram $D$, and let $ (X,\lhd, \kappa,\rho)$ be a symmetric skew-rack.
Choose orientations $o$ on each component of $D$, and denote by $D^o$ the diagram with the orientations.
In this paper, a {\it semi-arc of $D$} means a path from a crossing to the next crossing along the diagram.
Then, {\it an $X$-coloring} is a map $\mathcal{C}: \{$semi-arc of $D \} \ra X$ such that, for every crossing $\tau $ of $D$, the semi-arcs around $\tau$ satisfy
$\mathcal{C}(\mathcal{\gamma_{\tau}})= \kappa( \mathcal{C}(\mathcal{\beta_{\tau}}) )$ and $\mathcal{C}(\mathcal{\delta_{\tau}})= \mathcal{C}(\mathcal{\alpha_{\tau}}) \lhd \mathcal{C}(\mathcal{\beta_{\tau}})$, where $\alpha_{\tau}, \beta_{\tau}, \gamma_{\tau} ,\delta_{\tau} $ are the semi-arcs as seen in Figure \ref{koutenpn}.
We denote by $\mathrm{Col}_X(D^o)$ the set of $X$-colorings of $D^o$.
Then, as a basic fact in quandle theory (see \cite{CEGN, FRS4}),
if two diagrams $D^o$ and $(D')^{o'}$ are related by a Reidemeister move of type II, type III, or a doubled type I,
then there exists a canonical bijection $\mathcal{B}_{D^o,(D')^{o'}}: \mathrm{Col}_X(D^o) \ra \mathrm{Col}_X((D')^{o'})$.
Moreover, thanks to the above axioms (SS1) and (SS2), if $D^{o'}$ is the same diagram $D$ with opposite orientation,
the correspondence $a \mapsto \rho(a)$ on the color of each semi-arcs on the opposite component
defines a bijection $\mathcal{B}_{D^o,D^{o'}}: \mathrm{Col}_X(D^o) \ra \mathrm{Col}_X(D^{o'})$. 
In particular, the set $\mathrm{Col}_X(D^o)$ up to bijections does not depend on the choice of orientations of $D$.
Accordingly, we sometimes use the expression $\mathrm{Col}_X(D) $ instead of $ \mathrm{Col}_X(D^o).$
Finally, we should emphasize that the map $ \mathrm{Tw}^{ \pm 1}$ in Proposition \ref{propss2} corresponds
with an addition of a $(\mp 1)$-framing in an arc as in the Reidemeister move of type I. 
\begin{figure}[htpb]
\begin{center}
\begin{picture}(100,26)
\put(-22,27){\large $\alpha_{\tau} $}
\put(-22,-13){\large $\gamma_{\tau} $}
\put(13,-13){\large $\delta_{\tau} $}
\put(13,25){\large $\beta_{\tau} $}

\put(-36,3){\pcc{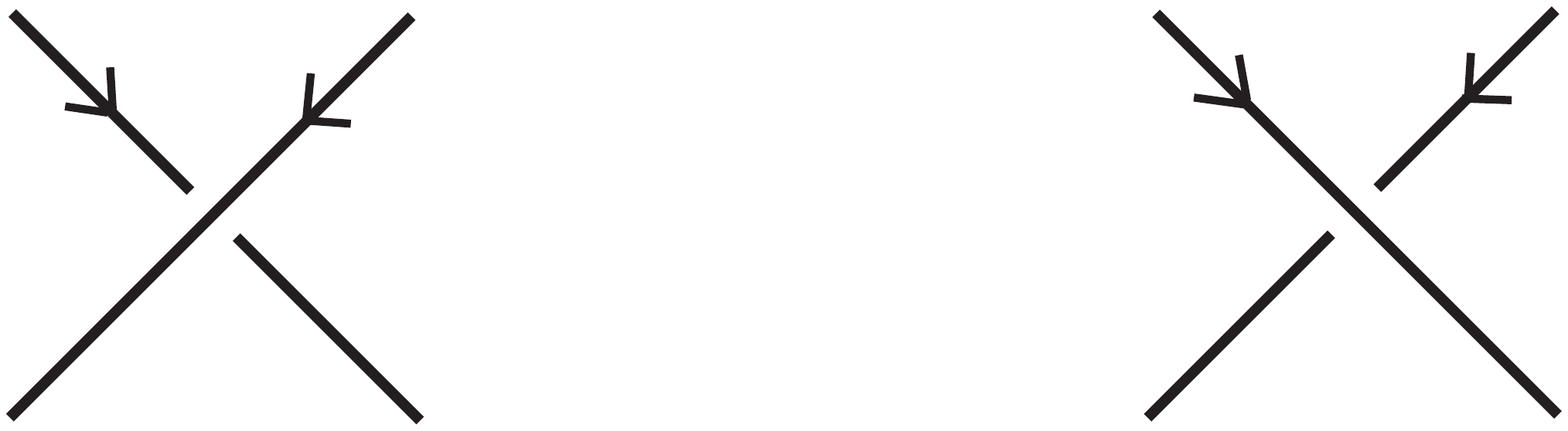}{0.27304}}

\put(131,-13){\large $\beta_{\tau} $}
\put(93,27){\large $\gamma_{\tau} $}
\put(93,-13){\large $\alpha_{\tau} $}
\put(131,27){\large $\delta_{\tau} $}
\end{picture}
\end{center}
\caption{\label{koutenpn} Positive and negative crossings, and eight semi-arcs with labeling.}
\end{figure}

Next, we observe cocycle invariants of a symmetric skew-rack $X$. 
According to \cite{FRS,CJKLS,CEGN},
a map $ \phi: X^2 \ra A$ for some abelian group $A$ is called {\it a birack 2-cocycle}, if
\begin{equation}\label{condition}\phi(a, b) + \phi(a \lhd b , c) = \phi ( a,\kappa (c)) +\phi( a \lhd \kappa (c) , b\lhd c), \ \ \ \ \ \phi ( b,c) = \phi ( \kappa (b),\kappa (c)), \end{equation} 
hold for any $ a,b,c \in X$.
Then, 
we define {\it the weight (of $\tau$)}, $ \Phi(\tau) $, with respect to a crossing $ \tau$ on $D $ to be $ \varepsilon_{\tau} \tau ( \mathcal{C}(\alpha_{\tau}) , \mathcal{C}(\beta_{\tau}) ) \in A $, where $\varepsilon_{\tau} $
is the sign $ \tau$ as in Figure \ref{koutenpn}; we further define $ \Phi_D (\mathcal {C}) \in A$ to be the sum $ \sum_{\tau} \Phi(\tau) $, where
$\tau $ runs over every crossing on $D $.
Then, as is known (see, e.g., \cite{CEGS, CEGN}), if two diagrams $D$ and $D' $ are related by a Reidemeister move of type II, type III or a doubled type I move,
then $ \Phi_{D'} \circ \mathcal{B}_{D^o,(D')^{o'} }= \Phi_D$ holds as a map $\mathrm{Col}_X(D^{o} ) \ra A $.
In other words, the map $\Phi_D:\mathrm{Col}_X(D^o) \ra A$ up to bijections is an invariant of framed links with orientation. As in \cite{CJKLS, CEGN}, we call the map $ \Phi $ {\it the (birack) cocycle invariant}.

Next, as an analogy of symmetric cocycle invariants in \cite{Kam, KO},
we discuss symmetric birack cocycles. We say a birack 2-cocycle $ \phi: X^2 \ra A$ to be
{\it symmetric} if
\[ \phi (a, b)= - \phi (a \lhd b, \rho(b)) =- \phi (\rho(a), \kappa (b)) \in A, \]
for any $a,b \in X$. Then, by a similar discussion to \cite[Theorem 6.3]{KO}, we can easily check that the weight $ \Phi(\tau)$
does not depend on the choice of orientations of $D$; neither does the map $\Phi_{D}:\mathrm{Col}_X(D^o) \ra A$ up to bijections.
In conclusion, the cocycle invariant $\Phi_{D}:\mathrm{Col}_X(D^o) \ra A$ up to bijections is an invariant of framed links.


Finally, we briefly review surgery on links and Fenn-Rourke moves \cite{FR}.
Let us regard a framed link diagram as the surgery on the framed link in the 3-sphere.
As folklore, every closed 3-manifold $M$ can be expressed as
the result of $S^3$ of surgery on a framed link.
Furthermore, it is shown (see \cite{FR}) that two framed links in $S^3$ have orientation-preserving homeomorphic results of surgery if and only if the framed links are related by a sequence of Fenn-Rourke moves and isotopies, where
{the Fenn-Rourke move} is an operation between the framed links shown in Figure \ref{kdn66}.
Throughout this paper, for a framed link diagram $D$ of a link $L$, we denote by $M_D$ the result of surgery of $S^3$ on $L$.

\section{Topological invariants from skew racks with Property FR.}\label{SCGSS}
It is reasonable to find appropriate skew-racks, which yield birack cocycle invariants that are invariant with respect to the Fenn-Rourke moves.
The purpose of this section is to define skew-racks with Property FR, and colorings of closed 3-manifolds.

For $\varepsilon \in \{ \pm 1\}$ and $ a_1,\dots, a_n \in X$,
let us consider the bijection $A_{a_1, \dots, a_n }: X \ra X$ that sends $x$ to $(
\cdots((x \lhd a_1) \lhd a_2 ) \lhd \cdots ) \lhd a_n$,
and define the following subsets:
\[\mathrm{Ann}^{+1}(A_{a_1, \dots, a_n }) := \{ x \in X \mid \kappa^{n+1} (x) = A_{a_1, \dots, a_n } (x) \lhd \kappa^{n+1} (x) \}, \]
\begin{equation}\label{pp033} \mathrm{Ann}^{-1}(A_{a_1, \dots, a_n }) := \{ x \in X \mid \kappa^{n+1} (x) \lhd  \kappa \bigl( A_{a_1, \dots, a_n } (x) \bigr) = A_{a_1, \dots, a_n } (x) \}. \end{equation}
As the case $n=0$, we define $\Ann (X) $ to be the subset $\{ x \in X \mid x \lhd \kappa (x) = \kappa(x) \}$

\begin{defn}\label{def4} A symmetric skew-rack $(X,\lhd, \kappa, \rho)$ is said to have {\it Property FR} if it satisfies the followings:
\begin{enumerate}[(FR1)]
\item The subset $\Ann (X) $ is not empty, and 
bijective to $\mathrm{Ann}^{ \varepsilon} (A_{a_1, \dots, a_n }) $ for any $ a_1,\dots, a_n \in X$ and $ \varepsilon \in \{ \pm 1\}$.
\item For any $ a_1,\dots, a_n \in X $ and $x \in \mathrm{Ann}^{+1}(A_{a_1, \dots, a_n }) , y \in \mathrm{Ann}^{-1}(A_{a_1, \dots, a_n }) $, the equalities
\begin{equation}\label{pp0} \kappa^{n+i} (a_i)= A_{a_1, \dots, a_n } \bigl( \kappa^{i+1} (a_i) \lhd x \bigr), \end{equation}
\begin{equation}\label{pp03}
 \kappa^{n+i}(a_i) \lhd \kappa^{ n+1 }(y) = A_{a_1 \kappa (y), a_2 \lhd \kappa^2 (y), \dots, a_n \lhd \kappa^n (y) }(\kappa^{i+1} (a_i) ) \end{equation}
hold, where $i \leq n$ is arbitrary.
\end{enumerate}
\end{defn}
Let us analyze the set of colorings of skew-racks with Property FR.

\begin{thm}\label{m3nthm4}
Let $(X,\lhd, \kappa, \rho)$ be a symmetric skew-rack with Property FR.
Suppose that two framed link diagrams $D$ and $D'$ are related by a Fenn-Rourke move as in Figure \ref{kdn66},
and take orientations on $D$ and $D'. $

Then, there is a bijection $\mathcal{B}: \mathrm{Col}_X(D) \ra \mathrm{Col}_X(D') \times \Ann (X) $.


In particular, if $X$ is of finite order, then the rational number $ | \mathrm{Col}_X(D)|/|\mathrm{Ann}(X) |^{\# D} \in \Q $ givesrise to a topological invariant of closed 3-manifolds.
\end{thm}
\begin{proof} 
Take arcs $\gamma, \alpha_i$'s and $\beta_i$'s as in Figure \ref{kdn66}. We may assume that the framing of the arc $\gamma$ is $+1$
since the same proof similarly runs well in the negative case.
Furthermore, by the properties of good involutions, the coloring conditions are independent of the choices of orientations of $D$.
Thus, we fix orientations of $ D$ and $D'$ as in Figure \ref{kdn66}.

\begin{figure}[htpb]
\begin{center}
\begin{picture}(100,96)

\put(-142,63){\Large $D$}
\put(-142,32){\large $\gamma$}
\put(92,63){\Large $D'$}

\put(-112,78){\large $\alpha_1 $}
\put(-88,78){\large $\alpha_2 $}
\put(-32,78){\large $\alpha_n $}
\put(-112,-13){\large $\beta_1 $}
\put(-88,-13){\large $\beta_2 $}
\put(-32,-13){\large $\beta_n $}
\put(-9,33){\large $\pm 1$-framing}
\put(139,35){\large $\mp 1$-framing}
\put(74,33){\large $\longleftrightarrow$}

\put(116,43){\pcc{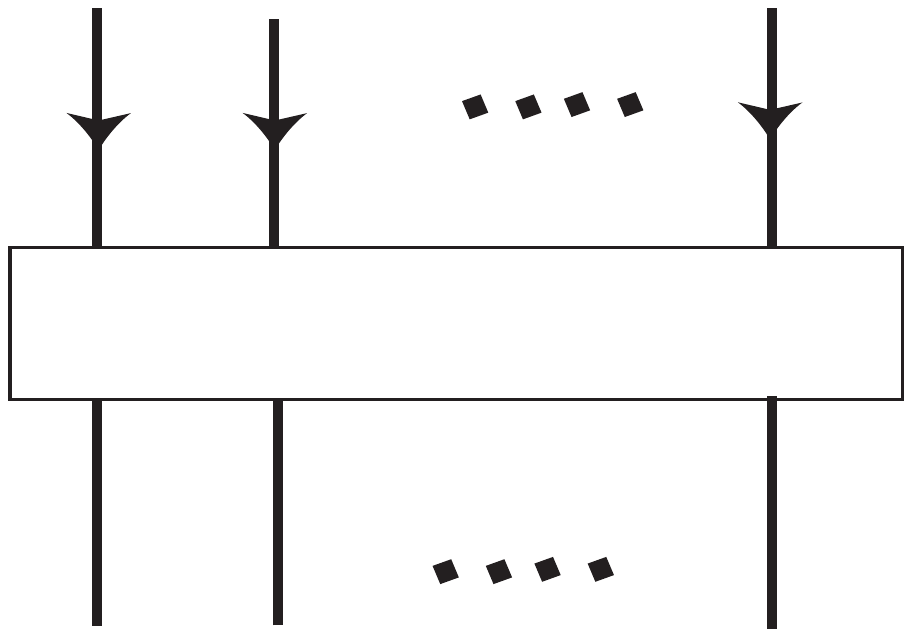}{0.41304}}
\put(-136,33){\pcc{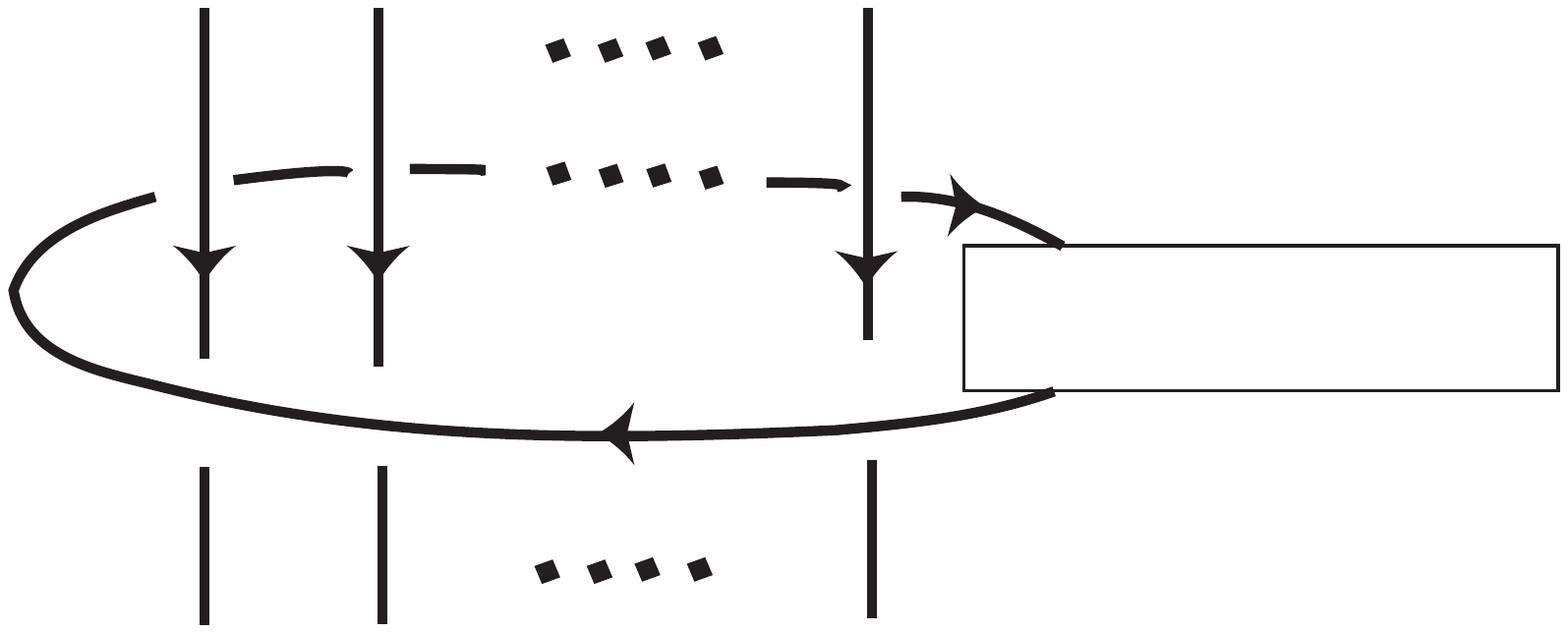}{0.41304}}

\put(125,78){\large $\alpha_1'$}
\put(208,78){\large $\alpha_n '$}
\put(149,78){\large $\alpha_2' $}
\put(125,-13){\large $\beta_1' $}
\put(208,-13){\large $\beta_n' $}
\put(149,-13){\large $\beta_2' $}

\end{picture}
\end{center}
\caption{\label{kdn66} Fenn-Rourke moves, and labeled semi-arcs.}
\end{figure}

Given an $X$-coloring $\mathcal{C} \in \mathrm{Col}_X(D^o)$, we will construct another $X$-coloring of $(D')^{o'}$.
Define $a_i :=\mathcal{C}( \alpha_i) $, $b_i :=\mathcal{C}( \alpha_i) $, and $x:= \mathcal{C}( \gamma) $.
The coloring condition on the arc $\gamma$ 
is $ \kappa^{n+1} (x) = A_{a_1, \dots, a_n }(x) \lhd \kappa^{n+1} (x)$; hence, $x \in \Ann (A_{a_1, \dots, a_n }) $.
Therefore, thanks to \eqref{pp0},
the map which sends $\alpha_i' $ to $ \mathcal{C}( \alpha_i) =a_i$ defines an $X$-coloring $\mathcal{C}' \in \mathrm{Col}_X((D')^{o'}) $.

Conversely, given an $X$-coloring $\mathcal{C}'$ of $(D')^{o'} $ and $x \in \mathrm{Ann}^{+1}(A_{a_1, \dots, a_n }) \neq \emptyset $,
we can define an $X$-coloring $\mathcal{C}$ of $ D$ that sends $\alpha_i $ to $ \mathcal{C}( \alpha_i') $ and $\gamma$ to $x$.
Thus, the correspondence $ \mathcal{C} \mapsto \mathcal{C}'$ gives the required bijection $\mathcal{B}$.
\end{proof}
Before going to the next section, we now discuss the triviality of the invariants up to link homotopy.
For this, consider the permutation group, $\mathrm{Bij}(X)$, of a skew-rack $X$,
and define a subgroup generated by the following set:
\begin{equation}\label{popo}\{ ( \kappa (\bullet) \lhd a) \mid a \in X \} \cup 
\{ ( \bullet \lhd^{\epsilon_1 } a_1) \lhd^{\epsilon_2} a_2 \mid a_i \in X, \epsilon_i \in \{ \pm 1\} \}. \end{equation}
The subgroup canonically has the right action on $X$.
We denote the subgroup by $\mathrm{Inn}^{\rm even}_{\kappa }(X)$.
We say a skew-rack $(X,\lhd, \kappa)$ with Property FR to be {\it $f$-link homotopic},
if $ x \lhd^{\varepsilon} \kappa (x) = x \lhd^{\varepsilon } (x \cdot g)$ holds for any $x \in X, g \in \mathrm{Inn}^{\rm even}_{\kappa }(X) ,\varepsilon \in \{\pm 1 \} $.
\begin{prop}\label{pm3nthm4}
Suppose that a symmetric skew-rack $(X,\lhd, \kappa, \rho)$ with Property FR
is $f$-link homotopic.
Then, if two framed link diagrams $D$ and $D'$ are transformed by an operation in Figure \ref{kdn2}, 
then there is a bijection $\mathcal{B}_{f}: \mathrm{Col}_X(D^o) \ra \mathrm{Col}_X((D')^{o'}) $.
\end{prop}
\begin{proof} For a coloring $\mathcal{C} \in \mathrm{Col}_X(D^o) $, take $a \in X$ such that $ \mathcal{C}(\alpha) =\kappa(a)$. Since $\alpha$ and $\beta$ lie on the same link-component,
there is $g \in \mathrm{Inn}^{\rm even}_{\kappa }(X) $ such that $\mathcal{C}(\beta) = a \cdot g $ from the definition \eqref{popo}.
Then, by the rule of colorings, we have
\[ \mathcal{C}(\gamma) = \mathrm{Tw}^{-1}( \kappa (a \cdot g)) = (a \cdot g ) \lhd \kappa ( a \cdot g) , \ \ \ 
\mathcal{C}(\delta) = \mathrm{Tw}^{-1}( \kappa (a ) ) \lhd ( a \cdot g) = (a \lhd \kappa (a) ) \lhd (a \cdot g). \]
Since $X$ is $f$-link homotopic, $\mathcal{C}(\delta) = a $ and $\mathcal{C}(\gamma) = (a \cdot g) \lhd^{-1}a$.
Thus, we can define another coloring $\mathcal{B}_{f} (\mathcal{C} ) $ by
$ \mathcal{B}_{f} (\mathcal{C} )(\alpha) =a$ and $ \mathcal{B}_{f} (\mathcal{C} )(\beta) =a \cdot g $. Since $\mathcal{C}(\gamma) = (a \cdot g) \lhd^{-1}a = \mathcal{B}_{f} (\mathcal{C} ) (\gamma) $ and $\mathcal{C}(\delta)= a = \mathcal{B}_{f} (\mathcal{C} ) (\delta)$,
the map $\mathcal{B}_{f}: \mathrm{Col}_X(D^o) \ra \mathrm{Col}_X((D')^{o'}) $ is bijective, as required.
\begin{figure}[htpb]
\begin{center}
\begin{picture}(100,64)

\put(-142,48){\Large $D^o$}
\put(22,48){\Large $(D')^{o'}$}

\put(-102,52){\large $\alpha$}
\put(-42,52){\large $\beta $}
\put(-93,-10){\large $\gamma$}
\put(-39,-12){\large $\delta $}
\put(-139,9){\normalsize $(-1)$-framing}
\put(-61,9){\normalsize $(-1)$-framing}
\put(4,23){\large $\longleftrightarrow$}

\put(-146,23){\pcc{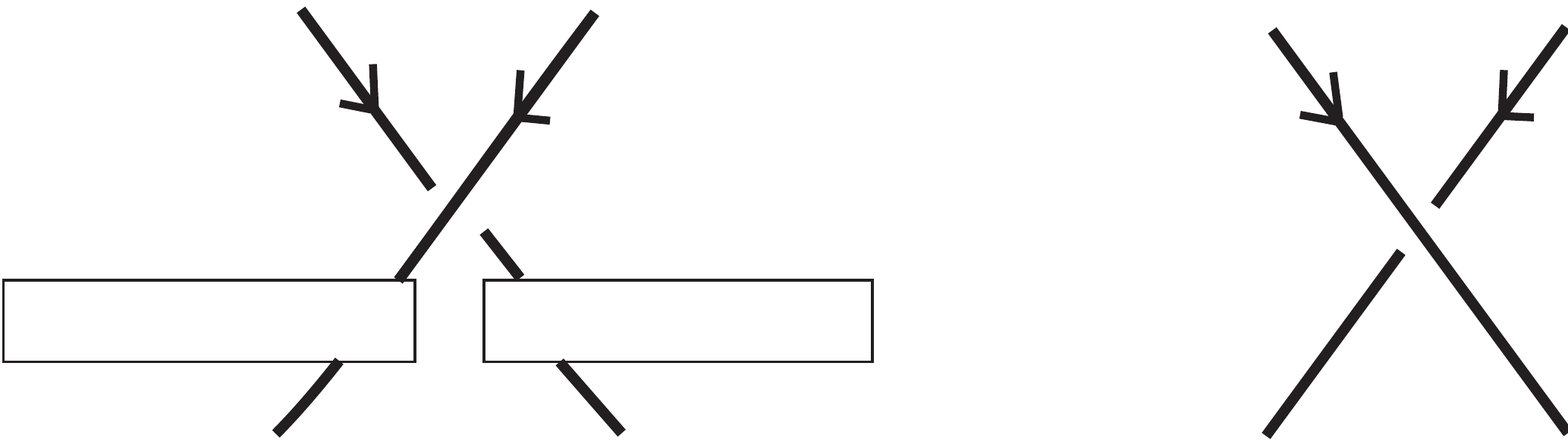}{0.41304}}

\put(106,52){\large $\beta'$}
\put(67,52){\large $\alpha' $}
\put(106,-5){\large $\delta '$}
\put(43,-5){\large $\gamma'$}

\end{picture}
\end{center}
\caption{\label{kdn2} The diagrams $D, D'$, where all the semi-arc lies in a link component. }
\end{figure}
\end{proof}
Since many
3-manifolds can be expressed as the results from $S^3$ of surgery along some framed knots,
in order to get non-trivial colorings, we shall consider skew-racks, which are not $f$-link homotopic.

\section{Examples of skew-racks with property FR from groups}\label{SCG2}
We give examples of skew-racks with property FR.
Throughout this section, we fix a group $G$, an automorphism $\kappa: G \ra
G $ satisfying $\kappa \circ \kappa = \mathrm{id}_G$,
and a map $\delta : G \ra G$ satisfying $\kappa \circ \delta = \delta \circ \kappa$.
Consider the binary operation $\lhd : G \times G\ra G$ defined by $ x \lhd y = \kappa (x) \delta (y).$
Then, the twisting map $\mathrm{Tw}$ in (SS3) is given by $\mathrm{Tw}(g)= g\delta (g)^{-1}. $
\begin{lem}\label{exa3434}
These operations $(\lhd, \kappa)$ define a skew-rack of $X=G$ if and only if
the following holds for any $x,y \in G$:
\begin{equation}\label{key1}
\delta(x) \delta(y)= \delta(y) \delta( x \delta(y)) \in G . \end{equation}
Let $\rho: G \ra G$ be a good involution.
Furthermore, assume that the image $\mathrm{Im}(\delta) \subset G$ is
a subgroup of $G$, and that the cardinality of the preimage $\delta^{-1}(d) $
is constant for any $d \in \mathrm{Im}(\delta)$.
Then, the symmetric skew-rack on $X=G$ has Property FR.

In addition, if the subgroup $ \mathrm{Im}(\delta)$ is commutative, then the skew-rack is $f$-link homotopic.
\end{lem}
\begin{proof}Since the former part is shown by direct computation,
we will show only the remaining claims.
We now analyse the set $\mathrm{Ann}^{\varepsilon}(A_{a_1, \dots, a_n})$ in \eqref{pp033}.
First suppose $\varepsilon = +1$ for simplicity.
Then, the condition
$\kappa^{n+1} (x) = A_{a_1, \dots, a_n } (x) \lhd \kappa^{n+1} (x)$
is equivalent to 
\begin{equation}\label{pp13354} \delta(\kappa^n(a_1)) \delta(\kappa^{n-1}(a_2)) \cdots \delta( \kappa(a_n)) \delta(\kappa^{n+1} (x)) =1 . \end{equation}
Since $\mathrm{Im}( \delta ) $ is a subgroup of $G$ by assumption, the set $\mathrm{Ann}^{+1}(A_{a_1, \dots, a_n})$ is non-empty.
Moreover, by the second assumption, the cardinality of
$\mathrm{Ann}^{+1}(A_{a_1, \dots, a_n}) $ does not depend on the choice of $ a_1, \dots, a_n$, that is, $X$ satisfies (FR1).
Concerning (FR2), the equality \eqref{pp0} is shown by
\[ A_{a_1, \dots, a_n} (\kappa^{i-1} (a_i) \lhd x) = \kappa^{n+i}(a_i) \delta(\kappa^{n+i} (x)) \delta(\kappa^{n+i+1}(a_1)) \delta(\kappa^{n+i}(a_2)) \cdots \delta( \kappa^i (a_n)) = \kappa^{n+i+1}(a_i) . \]
On the other hand, it is left to the reader to check \eqref{pp03} in the case $ \varepsilon = -1.$
Hence, $X$ has Property FR, as required.

Finally, we will show the last statement.
From the definition of the subgroup $ \mathrm{Inn}^{\rm even}_{\kappa}(X) $,
any $ g \in \mathrm{Inn}^{\rm even}_{\kappa}(X)$ and $a \in G$ uniquely admit
$b_1, \dots, b_{n} \in \mathrm{Im}( \delta )$ such that $ a \cdot g = a \delta (b_1) \cdots \delta (b_n) \in G $.
Since $ \mathrm{Im}(\delta)$ is commutative, \eqref{key1} means $\delta (a ) =\delta( a \delta (b))$.
Thus,
$$ (z \lhd^{\varepsilon} \kappa(a))\lhd^{- \varepsilon } (a \cdot g) = z \delta (a )^{\varepsilon}   \delta( a \delta (b_1) \cdots \delta (b_n) )^{- \varepsilon}  =z. $$
Therefore, the skew-rack is $f$-link homotopic by Proposition \ref{pm3nthm4}.
\end{proof}
We remark that the equality \eqref{key1} requires some conditions. For example, if $|G|> 1$, the map $\delta$ is not surjective.
In fact, if $ \delta $ is surjective, then \eqref{key1} with $ x=1$ is equivalent to $ z^{-1} \delta (1) z = \delta (z)$ for any $z \in G $, which means
that Im($ \delta )$ is a conjugacy class, and contradicts the surjectivity.
However, we will give some examples satisfying the conditions in Lemma \ref{exa3434}.
\begin{exa}\label{exa7}
First, we observe the case where $\delta$ is a group homomorphism. Then, we can easily check that the equality \eqref{key1} is equivalent to
that $\delta \circ \delta =0$ and the image $ \mathrm{Im}(\delta) $ is abelian.
If so, the cardinality of $\delta^{-1}(k)$ is constant; thus, if $X$ admits a good involution, then the symmetric skew-rack has Property FR, and is $f$-link homotopic by Lemma \ref{exa3434}.
\end{exa}
To avoid $f$-link homotopic skew-racks, we should focus on $\delta$ which is not a homomorphism.
\begin{exa}[Twisted conjugacy classes]\label{exa18}
Suppose a group automorphism $f : G \ra G$, and define $ \delta (x)= f (x^{-1}) x $.
Then, the equality \eqref{key1} is equivalent to $ \mathrm{Im}( \delta \circ \delta) = \{ 1_G \}$.
In general, it can be easily checked that, for any $g \in \mathrm{Im}( \delta) $, the preimage $\delta^{-1}(g)$ is bijective to
the fixed-point subgroup $\{ h\in G \mid f ( h)=h\}$; see, e.g., \cite{BNN}.
Thus, to apply Lemma \ref{exa3434}, the remaining point is to analyze the situation such that the image $ \mathrm{Im}( \delta)$ is a subgroup.

The image $\mathrm{Im}( \delta) $ is sometimes called {\it the twisted conjugacy classes} or {\it Reidemeister conjugacy classes}.
The papers \cite{BNN,GN} investigate some conditions to require that $ \mathrm{Im}( \delta)$ is a subgroup and $ \mathrm{Im}( \delta \circ \delta) = \{ 1_G \}$.
However, many examples in the papers satisfy that $ \mathrm{Im}(\delta)$ is commutative
Thus, it seems hard to find examples of pairs $(G,f)$ satisfying that the resulting skew-racks are not $f$-link homotopic.
\end{exa}
\begin{exa}\label{exa28}
Take a group $K$ with a normal subgroup $N \unlhd K$, and an involutive automorphism $f: K \ra K$ satisfying $f(N) \subset N$.
Let $G$ be $K \times N$ and $\kappa$ be $ f \times f$.
Define $\delta (x ,y) $ to be $ ( x^{-1} yx,1)$, where $x,y \in K$.
Next, we will check the conditions in Lemma \ref{exa3434}. The check of \eqref{key1} is obvious.
Since $N= \{ b^{-1} ab \mid a\in N ,b \in K\}$, the image of $\delta $ is $N \times 1$ as a subgroup of $G$.
Moreover, for any $(k,1) \in K \times 1$, the preimage $\delta^{-1}(k,1)$ is equal to $ \{ ( y^{-1} k y, y) \in G \mid y \in K \} $ which is bijective to $K$. In conclusion,
the symmetric skew-rack on $G$ has Property FR, by Lemma \ref{exa3434}.
For example, if $N=K$, the skew-rack on $G$ is equal to that in Example \ref{ss2} exactly.
\end{exa}


Finally, we will compute some colorings using the above skew-racks with Property FR.
\begin{exa}\label{exa3332}
For natural numbers $n, m \in \mathbb{N}$,
we first compute colorings of the lens space $L(nm -1, n)$.
Here, let $X=G$ be a skew-rack with good involution
and satisfy the conditions in Lemma \ref{exa3434}.
Let $D$ be the Hopf link with framing $(n,m)$. Then, $M_D$ is knwon to be $L(nm -1 ,n)$.
We fix two semi-arcs $\alpha,\beta$ in each link-component on $D$.
Then, from the definition of colorings, a coloring $\mathcal{C} \in \mathrm{Col}_X(D)$ satisfies 
\begin{equation}\label{pp7} \kappa ( \mathcal{C} (\alpha) \lhd \mathcal{C} (\beta) )= \mathrm{Tw}^n( \mathcal{C} (\alpha)) , \ \ \ \ \kappa ( \mathcal{C} (\beta) \lhd \mathcal{C} (\alpha) )= \mathrm{Tw}^m( \mathcal{C} (\beta)) .\end{equation}
Conversely, every $ a, b \in X$ satisfying $\kappa ( a\lhd b)= \mathrm{Tw}^n( a)$ and $ \kappa ( b\lhd a)= \mathrm{Tw}^m(b) $ yield a coloring of $D$. Since $\mathrm{Tw}^n(a)=a \delta(a)^{-n}$, notice that \eqref{pp7} is equivalent to the condition $\mathcal{C} (\beta) = \delta( \mathcal{C} (\alpha))^m $ and $\delta( \mathcal{C} (\alpha))^{nm-1} =1 $. Hence, $\mathrm{Col}_X(D)$ is bijective to
\begin{equation}\label{pp777} \{ (a,b) \in G^2 \mid \delta( a)^{nm-1} =1, \delta( b)= \delta( a)^{n} \} \stackrel{1:1}{\longleftrightarrow} \{ a \in G \mid \delta( a)^{nm-1} =1\} \times \delta^{-1}( 0) .\end{equation}
In particular, the set $\mathrm{Col}_X(D)$ depends only on $nm$, and thus, cannot classify the lens spaces of the forms $L(nm -1,n) $.
In contrast, we later compute some cocycle invariants, which can distinguish among some lens spaces (see Example \ref{exa18}). 
\end{exa}
\begin{exa}\label{exa3355}
Next, we will observe that the sets of colorings of integral homology 3-spheres seem strong invariants,
where we consider the skew-rack on $X= K \times K$ in Example \ref{ss2}.
Let $D_n^{\pm }$ be the $(2,n)$-torus knot with framing $\pm 1$. Then, the resulting 3-manifold $M_{D_n^{\pm }}$ is
the Brieskorn 3-manifold of the form $\Sigma(2,n, 2n \mp 1 )$, as an integral homology 3-sphere.
Then, for a concrete group $K$, it is not so hard to determine the set $\mathrm{Col}_X(D_q^{\pm }) $ with the help of the computer program. 
For example, we give a list of some computation of $|\mathrm{Col}_X(D_q^{\pm }) | $; see Table \ref{444}.
\end{exa}
\begin{rem}\label{exa33355} In this example, we focus on non-abelian groups $K$. In fact,
if $K$ is abelian, $(x,a)\lhd (y,b)=(x,a)$; hence, the coloring conditions are trivial. Thus,
considering the linking matrix of $D$, we can easily find a 1:1-correspondence $ \mathrm{Col}_X(D) \simeq \Hom(H_1(M;\Z), K) \times K^{\sharp D}$.
\end{rem}

\begin {table}
\begin {tabular}{l | cccccccccccc}
$ p $ & $|\mathrm{Col}_X(D_3^{+} )|$ & & $|\mathrm{Col}_X(D_3^{-} )|$ & & $|\mathrm{Col}_X(D_5^{+} )|$ & & $|\mathrm{Col}_X(D_5^{-} )|$ & & $|\mathrm{Col}_X(D_7^{+} )|$ & & $|\mathrm{Col}_X(D_7^{-} )|$\\ \hline
3 & $ |K|$ & & $ |K|$ & & $ |K|$ & & $|K|$ & & $ |K|$ & & $ |K|$ \\ \hline
5 & $ 121|K|$ & & $ |K|$ & & $ 121|K|$ & & $ |K|$ & & $ |K|$ & & $ |K|$ \\ \hline
7 & $ |K|$ & & $ 337|K|$ & & $ |K|$ & & $ |K|$ & & $ |K|$ & & $ 57|K|$ \\ \hline
11 & $2641 |K|$ & & $ |K|$ & & $ 2641|K|$ & & $ 2641|K|$ & & $ |K|$ & & $ |K|$ \\ \hline
13 & $|K|$ & & $ 6553|K|$ & & $ |K|$ & & $ |K|$ & & $ |K|$ & & $ |K|$ \\ \hline
\end {tabular}
\caption{The cardinality of $\mathrm{Col}_X(D_n^{\pm 1 }) $ for some $p,n$. Here $|K|=|\mathrm{SL}_2(\F_p)|=p(p^2-1).$}
\label{444}
\end {table}

\section{Cocycle invariants of 3-manifolds}\label{SCG211}
As seen in \cite{CEGS, CEGN, Nosbook}, there are some procedures to concretely find symmetric birack 2-cocycles. However, the conditions of 2-cocycles to require the invariance with respect to Fenn-Rourke moves seem strong.
However, in this section, we discuss 2-cocycle invariants to obtain 3-manifold invariants.
Throughout this section, we suppose a symmetric skew-rack $X$ with Property FR, and
a map $\phi$ from $X^2$ to a commutative ring $A $.

We first introduce Property FR of birack 2-cocycles as follows:
\begin{defn}\label{def8}
Let us recall the bijection $\mathcal{B}$ in Theorem \ref{m3nthm4}, and denote by $0_A$ the constant map to $A$ whose image is zero.
A symmetric birack 2-cocycle $\phi:X^2 \ra A$ satisfies {\it Property FR}, if $\Phi_D = (\Phi_{D '} \times 0_A)\circ \mathcal{B}$ holds for any diagrams $D$ and $D'$ in Figure \ref{kdn66}. Here, $\Phi_C$ is the cocycle invariant explained in \S \ref{gg2}. 

Furthermore, $\phi$ is said to be {\it $f$-link homotopic} if $X$ is $f$-link homotopic and the following holds for any
$a \in X$ and $g \in \mathrm{Inn}^{\rm even}_{\kappa}(X)$:
\begin{equation}\label{kkk4} \phi( ( a \cdot g) \lhd \kappa(a), a \cdot g ) + \phi( a, a\lhd \kappa(a)  ) = \phi(\kappa (a) , a \cdot g ) + \phi ( ( a \cdot g) \lhd \kappa(a), a  ) .\end{equation}
\end{defn}
We will see (Proposition \ref{prop8}) that symmetric birack 2-cocycles with Property FR yield topological invariants of closed 3-manifolds.
Take two maps $ F: Y \ra A$ and $ G: Z \ra A $, where $Y$ and $Z$ are some sets.
We call $F$ {\it an FR-stabilization} of $G$, if there is a bijection $B : Z \rightarrow Y \times \mathrm{Ann}(X) $ such that $g \circ B^{-1}= f\times 0_A $. More generally, $F $ and $G$ are {\it FR-equivalent}, if $F$ and $G$ are related by a finite sequence of FR-(dis-)stabilizations.
Then, the following proposition is almost obvious, by definitions.
\begin{prop}\label{prop8}
Let $\phi$ be a symmetric birack 2-cocycle with Property FR. Then,
the correspondence $D \mapsto \Phi_D$ up to FR-equivalent relations is an invariant of closed 3-manifolds.

Moreover, if $X$ and $\phi$ are $f$-link homotopic, and if $D$ and $D'$ are related by the operation in Figure \ref{kdn2}, then $\Phi_D=\Phi_{D'} \circ B_f$, where $B_f$ is the bijection $\mathrm{Col}_X(D) \ra \mathrm{Col}_X(D' ) $ in the proof of Proposition \ref{pm3nthm4}.
\end{prop}
To conclude, to obtain 3-manifold invariants, it is important to find symmetric birack 2-cocycles with concrete expressions.
For this, let us discuss Lemmas \ref{thm3338} and \ref{thm33389} below.
Let $\X $ be $ X \times A$. 
Define $ \tilde{\lhd} :\X \times \X \ra \X$ by
$$ (x, a)\tilde{\lhd} (y,b)= ( x \lhd y, a + \phi(x,y) ), \ \ \ \ \ \ \ \ (x,y \in X, a,b \in A ), $$
and $\tilde{\kappa} : \X \ra \X$ by $ \tilde{\kappa}(x,a) = (\kappa (x), -a)$. 
\begin{lem}[{cf. \cite[Section 3]{CEGS}}]\label{thm3338}
These maps $\tilde{\lhd},\tilde{\kappa}$ define a skew-rack on $\X= X \times A $,
if and only if $\phi$ is a birack 2-cocycle.
\end{lem}
This can be immediately shown by direct computation. Furthermore, we now show 
\begin{lem}\label{thm33389}
Suppose a birack 2-cocycle $\phi$ satisfying $\phi(a,b)=-\phi(\rho(a) ,\kappa(b))$.
Then, the map $\overline{\phi}:X^2 \ra A$ that sends $(a,b)$ to $\phi(a,b) -\phi(a\lhd b, \rho(b) ) $ is a symmetric birack 2-cocycle.
\end{lem}
\begin{proof} It is easy to check $\overline{\phi}(a,b)+ \overline{\phi}(a \lhd b, \rho(b))=0$.
Thus, it remains to check the cocycle condition \eqref{condition} of $\overline{\phi} $. For this, we may show
\begin{equation}\label{hhh} \phi (a\lhd b, \rho(b)) +\phi( (a\lhd b) \lhd c,\rho(c))= \phi( a \lhd \kappa (c) , \rho(\kappa(c)))+ \phi((a\lhd b)\lhd c, \rho(b\lhd c)). \end{equation}
Replace $ (a\lhd b) \lhd c , \rho (b \lhd c) $ and $, \rho (\kappa (c))$ by $a,b$, and $c$ respectively.
Then, we can easily check that the replacement of \eqref{hhh} coincides with \eqref{condition}, which completes the proof.
\end{proof}

Using these lemmas, we will give some examples from skew-racks in Example \ref{exa28}. Let $N \unlhd  K$ be groups, and $f: K \ra K$ be an involutive automorphism satisfying $f(N) \subset N$.
Furthermore, take a normalized group 2-cocycle $\theta : K \times K \ra A$,
that is, $\theta$ satisfies
\[\theta(x,y)-\theta(x,yz)+\theta(xy,z)-\theta(y,z)=0,\ \ \ \ \theta(1_K,x)= \theta(x,1_K ) =0 \in A,\]
for any $x,y,z\in K$.
Then, the product $\tilde{K} = K \times A$ has a group structure with operation $( (x,a), (y,b)) \mapsto (xy, a+b + \theta(x,y))$, as a central extension of $K$.
As is known in group cohomology, every central extension over $K$ with fiber $A$ can be expressed by the product for some $\theta.$
Then, from Example \ref{exa28}, we can define the symmetric skew-racks on $G=K\times N$ and $\tilde{G}= \tilde{K} \times \tilde{N}$,
which have Property FR.
Moreover, by the definition of $\lhd$ on $\tilde{G}$, we notice
\[ \bigl((x,a),(y,b) \bigr) \lhd \bigl((z,c),(w,d) \bigr)= \bigl( \tilde{\kappa}(x,a)(z^{-1}, -c-\theta(z,z^{-1}) )(w,d)(z,c), \ (y,b) \bigr)\]
\[=\bigl((f(x) z^{-1} wz , f(a) +d+ \theta ( f(x), z^{-1})+ \theta ( f(x) z^{-1}, w z)+\theta(w,z) -\theta(z,z^{-1}) ), (y,b) \bigr) \in \tilde{G}. \]
Inspired by Lemma \ref{thm3338}, we obtain a procedure of producing birack 2-cocycles as follows: 
\begin{thm}\label{lhm8}
Let $ \lambda : N \ra A$ be a group 1-cocycle. 
Then, the map 
\[\phi_{\lambda, \theta} : G^2=(K \times N) \times (K \times N) \ra A \]
\[ (x,y,z,w) \longmapsto \lambda(y) \bigl( \theta ( f(x), z^{-1})+ \theta ( f(x) z^{-1}, w z)+\theta(z,w) -\theta(z,z^{-1}) \bigr) \]
is a birack 2-cocycle of the skew-rack $G= K \times N$ in Example \ref{ss2}. If 
\[ \lambda(x) \theta(a,b )= \lambda(f(x)) \theta(f(a),f(b) ),\]
hold for any $a,b\in K,x \in N,$ then the condition in Lemma \ref{thm33389} is true. 
In particular, the cocycle $\overline{\phi_{\lambda, \theta} }$ mentioned in Lemma \ref{thm33389} is a symmetric birack 2-cocycle.
\end{thm}

In general, it seems hard to find group 2-cocycles $\theta$ such that the associated map $\phi_{\lambda, \theta} $
has Property FR.
However, when $K$ is a cyclic group,
we will give such examples of birack cocycles with Property FR. 
More precisely, by a direction computation, we can easily show that
\begin{prop}\label{prop28}
Let $p \in \Z$ be an odd prime.
Let $K=N= \Z/p$, and take $\varepsilon \in \{ \pm 1\} $ such that $f (x)= \varepsilon x, $
Let us define group cocycles $\lambda$ and $\theta $ by setting
\[ \lambda(x)=x, \ \ \ \ \theta (x,y) = \frac{ (x+ \varepsilon y)^p - x^p-(\varepsilon y)^p}{p} = \sum_{j: 1 \leq j <p } j^{-1} x^j (\varepsilon y)^{p-j},\]
respectively, where $x,y \in \Z/p$.
Then, $\overline{\phi_{\lambda, \theta} }(x,y,z,w)= 2y \theta (x,w) $, and the symmetric birack 2-cocycle $\overline{\phi_{\lambda, \theta} } $ has Property FR, and is $f$-link homotopic.
\end{prop}
\begin{exa}\label{exa18}
Let $ D$ be the Hopf link with framings $(n,m)$, as in Example \ref{exa3332}.
Recall that $M_D$ is the lens space $L(nm-1, m)$.
By \eqref{pp777}, if $ nm-1$ is divisible by $p$ and $K=\Z/p$, then $\mathrm{Col}_X(D)$ is bijective to $ (\Z/p)^2$.
In addition, we can easily show that the cocycle invariant $ \Phi_{ D} : (\Z/p)^2\ra \Z/p$ is equal to the correspondence $(x,y) \mapsto -m x^2 $, where we use the 2-cocycle $ \overline{\phi_{\lambda, \theta}} $ in Proposition \ref{prop28}. For instance, 
the invariant can distinguish between the lens spaces $L(11,1)$ and $L(1 1,3)$, which are not homotopy equivalent. 

More generally, consider the lens space $L( p,q)$, and a framed diagram $ D_{p,q}$ such that $M_{D_{p,q}} = L(p,q)$.
Then, with the help of a computer program, if $p,q <100$, it is not so hard to check that
the cocycle invariant $ \Phi_{ D_{p,q}} : (\Z/p)^{1+ \# D_{p,q}}\ra \Z/p$ is FR-equivalent to the map $\Z/p \ra \Z/p ; x \mapsto -q x^2 $.
\end{exa}
From this example, it is natural to pose a problem below, together with a relation to the Dijkgraaf-Witten invariant \cite[\S 6]{DW}.
For this, let us briefly review the invariant.
Fix a closed 3-manifold $M$ with the fundamental homology 3-class $[M] \in H_3(M;\Z) \cong \Z$.
Let $K$ be a group of finite order, and $ \psi: K^3 \ra A$ be a group 3-cocycle.
Denote by $ BK$ the classifying space of $K$ or the Eilenberg-MacLane space of type $(K,1)$,
and by $c_M: M \ra B\pi_1(M)$ be a classifying map.
Then, any group homomorphism $f : \pi_1(M) \ra K$ induces a continuous map $f_* : B \pi_1(M) \ra BK$.
Since the (co)-homology of $BK $ equals that of $K$, we can define the pullback $ (f_* \circ c_M)^*(\psi)$ as a 3-cocycle of $M$.
Then, {\it the Dijkgraaf-Witten invariant} is defined as the map
$$ \mathrm{DW}_{\psi}(M): \Hom(\pi_1(M),K) \lra A; \ \ \ f \longmapsto \langle (f_* \circ c_M)^*(\psi) , [M]\rangle,$$
where $\langle ,\rangle $ is the Kronecker map.

\begin{prob}\label{p3p64}As in Example \ref{ss2}, let $X$ be the symmetric skew-rack on $K \times K$.
Let $\lambda: K \ra A$ and $\theta :K^2 \ra A$ be group cocycles,
and $ \psi$ be the cup product $\lambda \smile \theta$ as a group 3-cocycle.
Let $D$ be a framed link diagram. 

Then, is there a bijection $\mathcal{B} : \mathrm{Col}_X(D) \simeq
\Hom(\pi_1(M),K)
\times \mathrm{Ann}(X)^{\sharp D}$? Furthermore, find a condition such that the birack 2-cocycle $\overline{\phi_{\lambda,\theta}}$ in \ref{lhm8} has Property FR, and FR-equivalence between the cocycle invariant $\Phi :\mathrm{Col}_X(D) \ra A $ and the Dijkgraaf-Witten invariant $ \mathrm{DW}_{\psi}(M_D) $. 
\end{prob}



If this problem is positively solved, we consequently obtain a diagrammatic computation of
the Dijkgraaf-Witten invariant via the cocycle invariants and Dehn surgery.


\section{Criteria for 3-manifolds which are not the surgery of any knot}\label{SCG211}
As an application of the cocycle invariant, we will give two criteria to detect some 3-manifolds, which are not the results of surgery of any knot in $S^3$; see, e.g., \cite{HKMP}, \cite[Section 7.1]{AFW} and references therein for the details of such 3-manifolds, and such other criteria.
For this, as in Example \ref{exa28}, we fix groups $N \unlhd K$, and $X= K \times N$ with $f= \id_K $;
recall that $X$ has a skew rack by $(x,a)\lhd (y,b)=( x y^{-1} b y ,a)$, and has Property FR.

\begin{prop}\label{pp15445}
Suppose $|K| < \infty $ and that a framed link diagram $D$ and a knot diagram of framing zero 
are related by a sequence of Fenn-Rourke moves and isotopy.
Then, the invariant $ |\mathrm{Col}_X(D)|/|K|^{\sharp D} \in \Q $ in Theorem \ref{m3nthm4} is larger than $|N|$ or equal to.
\end{prop}
\begin{proof}We may suppose that $D$ is a knot diagram of framing zero.
For the proof, it is sufficient to construct $|K \times N|$ colorings on $D$.
As in Figure \ref{logifig2}, take semi-arcs $\alpha_i$ and $\beta_i$ in $D$, and denote by $\varepsilon_i \in \{ \pm 1 \} $ the sign of the crossing between $\alpha_i$ and $\beta_i$. For $(g,h) \in K \times N$, we define $\mathcal{C}_{g,h} (\alpha_i)$ to be $ (h^{\sum_{j=1}^{i-1} \varepsilon_j  } g,h)\in X= K\times N$.
Since every $\beta_i$ lies on the same link component, $\mathcal{C}_{g,h} (\beta_i)=( h^{n_i}g,h)$ for some $n_i\in \Z$.
Hence, we can easily check that $\mathcal{C}_{g,h}$ defines an $X$-coloring as required.
\end{proof}
\begin{figure}[tpb]
\begin{center}
\begin{picture}(50,74)
\put(-68,25){\large \ \ \ \ $\alpha_1 $}
\put(-13,24){\large $\alpha_2 $}
\put(14,24){\large $\alpha_3 $}

\put(-66,37){\pcc{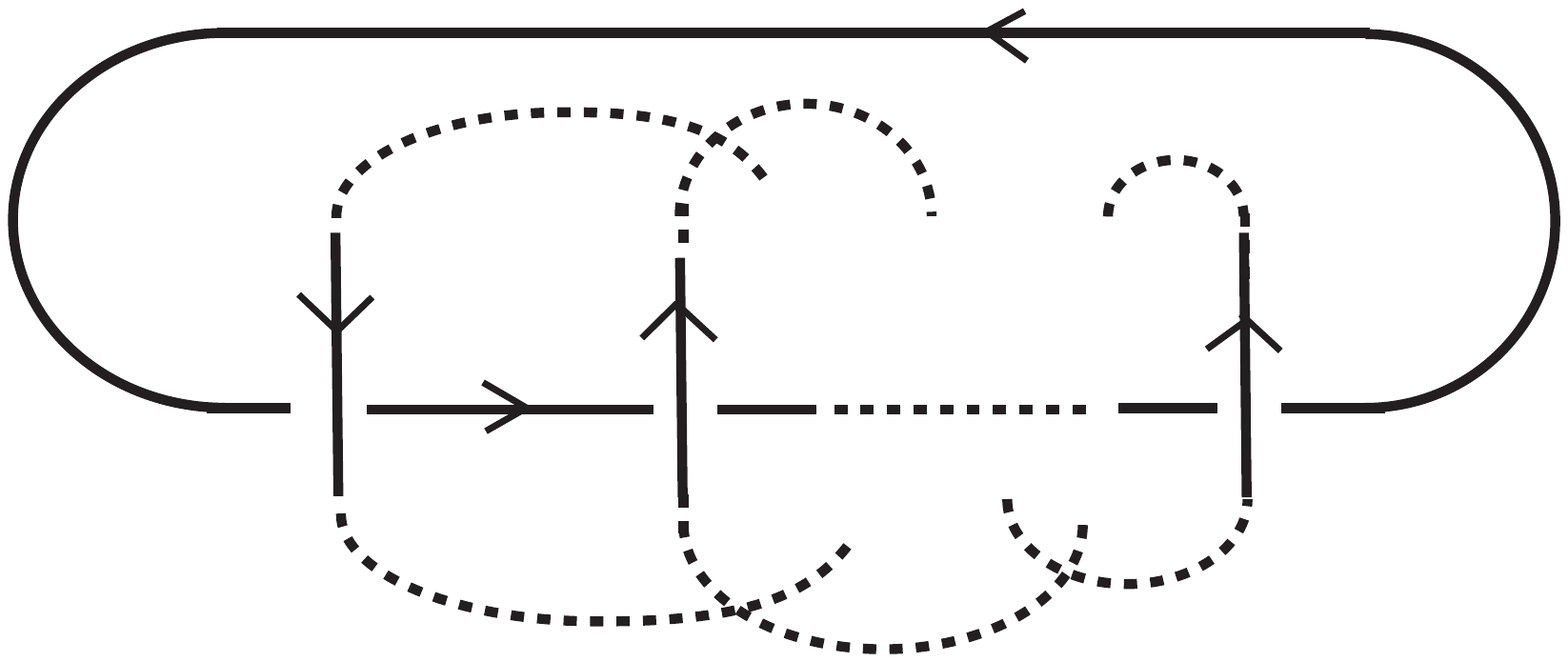}{0.34}}

\put(-39,50){\large $\beta_1 $}
\put(-6,50){\large $\beta_2 $}
\put(69,50){\large $\beta_{N_j} $}
\put(33,46){\large $\cdots $}
\end{picture}
\end{center}
\vskip -1.7pc
\caption{\label{logifig2} Semi-arcs $\alpha_i$ and $\beta_i$ in the knot diagram $D$. }
\end{figure}

As a special case, let $K=N = \Z/2$.
For $k_1,k_2,k_3 \in \Z/2$, let us define a map $\phi_{k_1,k_2,k_3} : X \times X \ra \Z/2$ by setting
\[ \phi_{k_1,k_2,k_3} ((x,a),(y,b)) = k_1 a +k_2 b +k_3 ab . \]
Then, by direct computation, it is not hard to show the following:
\begin{prop}\label{pp155}
The map is a symmetric birack 2-cocycle with Property FR, and is $f$-link homotopic.
Furthermore, if a framed link diagram $D$ is FR-equivalent to a knot diagram of framing zero, then the symmetric birack 2-cocycle invariant is trivial.
\end{prop}
However, unfortunately, the author does not find new examples of
framed link diagrams that are not FR-equivalent to any knot diagram of framing zero.
We end this paper by giving problems as future works.
\begin{prob} As applications from the propositions above,
find 3-manifolds that are not the surgery of any knot of framing zero. 
Establish stronger criteria than the propositions above, which are applicable to many framed link diagrams.
\end{prob}

\vskip 1pc

\normalsize

DEPARTMENT OF
MATHEMATICS
TOKYO
INSTITUTE OF
TECHNOLOGY
2-12-1
OOKAYAMA
, MEGURO-KU TOKYO
152-8551 JAPAN

\end{document}